\documentclass[a4paper, 12pt]{article}
\pdfoutput=1

\usepackage[T1]{fontenc}
\usepackage[utf8]{inputenc}

\usepackage{amsfonts,amsmath,amssymb,mathtools,microtype}
\usepackage[dvipsnames]{xcolor}
\usepackage[noBBpl]{mathpazo}

\DeclareFontFamily{U} {cmr}{}

\DeclareFontShape{U}{cmr}{m}{n}{
    <-6> cmr5
    <6-7> cmr6
    <7-8> cmr7
    <8-9> cmr8
    <9-10> cmr9
    <10-12> cmr10
    <12-> cmr12}{}

\DeclareSymbolFont{Xcmr} {U} {cmr}{m}{n}
\DeclareMathSymbol{\Delta}{\mathord}{Xcmr}{'001}
\DeclareMathSymbol{\Upsilon}{\mathord}{Xcmr}{'007}
\DeclareMathSymbol{\Omega}{\mathord}{Xcmr}{'012}
\DeclareMathSymbol{\Gamma}{\mathord}{Xcmr}{'000}
\DeclareMathSymbol{\Theta}{\mathord}{Xcmr}{'002}
\DeclareMathSymbol{\Lambda}{\mathord}{Xcmr}{'003}
\DeclareMathSymbol{\Xi}{\mathord}{Xcmr}{'004}
\DeclareMathSymbol{\Pi}{\mathord}{Xcmr}{'005}
\DeclareMathSymbol{\Sigma}{\mathord}{Xcmr}{'006}
\DeclareMathSymbol{\Phi}{\mathord}{Xcmr}{'010}
\DeclareMathSymbol{\Psi}{\mathord}{Xcmr}{'011}

\usepackage[labelsep = period, labelfont = bf, justification =
    centering]{caption}
\usepackage{float,graphicx,subcaption}
\DeclareCaptionSubType*[roman]{figure}

\usepackage{booktabs,makecell}
\usepackage{enumitem,mdwlist}
\setlist[itemize]{topsep=0ex,itemsep=0ex,parsep=0.4ex}
\setlist[enumerate]{topsep=0ex,itemsep=0ex,parsep=0.4ex}

\usepackage{pgf,tikz,ifthen,calc}
\usepackage{tkz-euclide}
\tkzSetUpPoint[fill=black, size = 3pt]
\tkzSetUpLine[color=black, line width=0.6pt]
\tkzSetUpCircle[color=black, line width=0.6pt]

\usepackage[left = 2.5cm, right = 2.5cm, top = 2.5cm, bottom = 2.5cm, headsep =
    12pt, headheight = 15pt]{geometry}
\usepackage{parskip}

\usepackage[hyphens]{url}
\usepackage[linktoc = all, hidelinks, colorlinks, unicode=true,
    pagebackref=true]{hyperref}
\usepackage[capitalise, compress, nameinlink, noabbrev]{cleveref}

\hypersetup{
    linkcolor={blue!70!black},
    citecolor={green!70!black},
    urlcolor={blue!70!black},
    pdftitle={Upper bounds for ordered Ramsey numbers of forests and
        bounded-degree graphs},
    pdfauthor={Lior Gishboliner and Xiangyu Li}
}

\usepackage{amsthm}
\theoremstyle{plain}
\newtheorem{thm}{Theorem}[section]
\newtheorem{claim}{Claim}[section]

\newtheorem{lem}[thm]{Lemma}

\newtheorem{que}[thm]{Question}

\newtheorem{construction}[thm]{Construction}

\theoremstyle{definition}

\crefname{subsection}{\S}{\S\S}
\newenvironment{proofclaim}[1][Proof of claim]
    {\begin{proof}[#1]}
    {\end{proof}}

\renewcommand*{\backref}[1]{}
\renewcommand*{\backrefalt}[4]{
    \ifcase #1 Not cited.\or \(\uparrow\)#2\else \(\uparrow\)#2\fi}

\renewcommand{\epsilon}{\varepsilon}
\renewcommand{\ge}{\geqslant}

\renewcommand{\geq}{\geqslant}
\renewcommand{\leq}{\leqslant}

\renewcommand{\subset}{\subseteq}

\DeclarePairedDelimiter{\ceil}{\lceil}{\rceil}

\newcommand*{\cA}{\mathcal{A}}

\newcommand{\defn}[1]{\textcolor{Maroon}{\emph{#1}}}

\newcommand{\Rord}{R_{<}}

\title{Upper bounds for ordered Ramsey numbers of forests and bounded-degree
    graphs}
\author{Lior Gishboliner\footnotemark[1] \and Xiangyu Li\footnotemark[2]}
\date{}
\begin{document}

\maketitle

\renewcommand{\thefootnote}{\fnsymbol{footnote}}
\footnotetext[1]{Department of Mathematics, University of
    Toronto, Canada. \emph{Email}: \href{mailto:lior.gishboliner@utoronto.ca}
    {\tt lior.gishboliner@utoronto.ca}. Research supported by an NSERC
    Discovery Grant.}
\footnotetext[2]{Department of Mathematics, University of Toronto, Canada.
    \emph{Email}: \href{mailto:xiangyuu.li@mail.utoronto.ca}{\tt
    xiangyuu.li@mail.utoronto.ca}. Research supported by a University of Toronto
    Excellence Award (UTEA).}

\renewcommand{\thefootnote}{\arabic{footnote}}

\begin{abstract}
    \noindent
    We prove the following two upper bounds for ordered Ramsey numbers:
    \begin{enumerate}[label=(\arabic*), leftmargin=*]
        \item Every ordered forest \(F\) on \(n\) vertices satisfies
            \[
                R_<(F,F)
                = O\bigl(n^{1+\lceil\log\chi_<(F)\rceil}\bigr).
            \]
            This in particular answers a question of Geneson, Holmes, Liu,
            Neidinger, Pehova and Wass.
        \item There is a function \(f\) such that, for every fixed ordered
            graph \(H\) with maximum degree at most \(\Delta\) and interval
            chromatic number at most \(k\), it holds that
            \[
                R_<(H,K_n)=O_H\bigl(n^{f(\Delta,k)}\bigr).
            \]
    \end{enumerate}
\end{abstract}

\section{Introduction}
An \defn{ordered graph} is a graph equipped with a linear order on its
vertices,\footnote{We will picture the vertices as ordered from left to right,
and thus speak of the leftmost/rightmost vertex, etc.} and copies are required
to preserve this order. For two ordered graphs \(G,H\), the \defn{ordered Ramsey
number} \(R_<(G,H)\) is the minimum \(N\) such that every red/blue edge-coloring
of \(K_N\) contains a red copy of \(G\) or a blue copy of \(H\). The
\defn{interval chromatic number} \(\chi_<(H)\) is the minimum number of
independent intervals partitioning \(V(H)\).

The systematic study of ordered Ramsey problems was initiated independently by
Balko et al.~\cite{BCKK} and Conlon et al.~\cite{CFLS}, although several
classical results (such as the Erd\H{o}s-Szekeres lemma \cite{ES}) can be cast
in this framework. For recent work on ordered Ramsey numbers,
see~\cite{BMSW,GJJ,KLPSY,Li2026}, and for a broader overview, see the survey of
Balko~\cite{Balko_survey}.

\subsection{Bounds on \texorpdfstring{\(R_<(F, F)\)}{R<(F, F)}}
Our first result answers a question of Geneson, Holmes, Liu, Neidinger, Pehova
and Wass~\cite{GHLNPW} (see also~\cite[Problem~4]{Balko_survey}), who
conjectured that \(R_<(P,P)=O(n^2)\) for every \(n\)-vertex ordered path \(P\)
with \(\chi_<(P)=2\). We prove this bound for all ordered forests of interval
chromatic number two, and extend our result by a routine induction to the case
of general ordered forests \(F\), where our bound on \(R_<(F,F)\) depends on the
interval chromatic number of \(F\). All logarithms in the paper are base 2.

\begin{thm}\label{cor:forest-ramsey}
    Let \(F\) be an \(n\)-vertex ordered forest, where \(n \geq 2\), and let
    \(H\) be an \(m\)-vertex ordered graph. Then
    \begin{equation}\label{eq:forest-off-diagonal}
        R_<(F,H)
        \leq 1 + (n-1)^{\ceil{\log \chi_<(H)}}(m-1).
    \end{equation}
    In particular, \(R_<(F,F) = O\bigl(n^{1+\ceil{\log\chi_<(F)}}\bigr)\).
\end{thm}

When \(\chi_<(F)=2\), the exponent in Theorem~\ref{cor:forest-ramsey} is optimal
up to an \(o(1)\) term. Indeed, Balko, Jel\'inek and
Valtr~\cite[Theorem~3.5]{BJV} (slightly improving an earlier result in
\cite{CFLS}) proved that there are arbitrarily large \(n\)-vertex ordered
matchings \(M\) with \(\chi_<(M)=2\) and
\[
    R_<(M,M)=
    \Omega\left(\frac{n^2}{\log^2 n}\right).
\]
For larger interval chromatic numbers, it would be interesting to determine the
optimal exponent in Theorem~\ref{cor:forest-ramsey}.

\begin{que}
    For each fixed \(k\geq3\), what is the smallest exponent \(c=c(k)\) such
    that every \(n\)-vertex ordered forest \(F\) with \(\chi_<(F)\leq k\)
    satisfies
    \[
        R_<(F,F)\leq n^{c+o(1)}?
    \]
\end{que}

Let \(K_{a, b}\) denote the ordered \(a \times b\) complete bipartite graph
where the part with \(a\) vertices comes entirely before the part with \(b\)
vertices in the vertex ordering. As we shall see, the main step in the proof of
Theorem~\ref{cor:forest-ramsey} is proving the case where \(H = K_{a,b}\), which
we now record separately:
\begin{thm}\label{thm:forest-bipartite}
    Let \(F\) be an \(n\)-vertex ordered forest and \(a,b\geq 1\) be integers.
    Then \(R_<(F,K_{a,b}) \leq 1+(n-1)(a+b-1)\).
\end{thm}

Theorem~\ref{thm:forest-bipartite} is sharp. Indeed, let \(P_n\) be the monotone
path on \(n\) vertices. Partition \(K_{(n-1)(a+b-1)}\) into \(n-1\) consecutive
blocks of size \(a+b-1\) and color edges within blocks blue and those between
blocks red. This coloring contains neither a red \(P_n\) nor a blue \(K_{a,b}\),
so \(R_<(P_n,K_{a,b})=1+(n-1)(a+b-1)\).

Moreover, Theorem~\ref{thm:forest-bipartite} may be considered as a
Ramsey-goodness result for ordered graphs. A classical result of Chv\'atal
\cite{Chvatal} states that in the unordered setting, \(R(T,K_m) = (v(T)-
\nolinebreak 1)(m-1)+1\) for every (unordered) tree.
Theorem~\ref{thm:forest-bipartite} suggests studying a similar problem for
ordered graphs and \(K_{a,b}\), i.e., determining for which trees \(T\) the
bound in Theorem~\ref{thm:forest-bipartite} is tight. Ramsey goodness with
respect to a complete graph \(K_m\) (rather than a complete bipartite graph) was
studied in \cite{BP}.

\subsection{Bounds on \texorpdfstring{\(R_<(H, K_n)\)}{R<(H, K\_n)}}
Our second result concerns the off-diagonal ordered Ramsey number of a fixed
ordered graph \(H\). The standard bound \(\Rord(H,K_n) \leq \Rord(K_{|H|},K_n)
=O_H(n^{|H|-1})\) has an exponent depending on \(|H|\). We show that the
exponent can instead be chosen to depend only on the maximum degree
\(\Delta(H)\) and the interval chromatic number \(\chi_<(H)\). This is an
ordered analogue of the classical off-diagonal bound of Alon, Krivelevich and
Sudakov~\cite[Theorem~4.1]{AKS}, proved using dependent random choice.

\begin{thm}\label{thm:bounded-degree}
    For all integers \(\Delta,k\geq1\), there exists an integer \(f(\Delta,k)\)
    such that the following holds. For every ordered graph \(H\) with
    \(\Delta(H)\leq\Delta\) and \(\chi_<(H)\leq k\), there is a constant \(C_H\)
    such that, for every integer \(n\geq1\),
    \[
        \Rord(H,K_n)
        \leq C_H n^{f(\Delta,k)}.
    \]
\end{thm}

For \(k=2\), Theorem~\ref{thm:bounded-degree} already follows, with the much
better exponent \(\Delta+o(1)\), from a result of Janzer, Janzer, Magnan and
Methuku~\cite[Theorem~1.7]{JJMM}. Their result implies that every fixed ordered
graph \(H\) with \(\chi_<(H)\leq 2\) and \(\Delta(H)\leq\Delta\) satisfies
\[
    \Rord(H,K_n)\leq n^{\Delta+o(1)}.
\]
Their argument uses random choices to find many intervals where the vertices of
\(H\) can be embedded greedily. We take a different approach and reduce the
ordered problem to an unordered one. Namely, we construct an unordered graph
\(G\) such that every ordering of \(V(G)\) contains an ordered copy of \(H\),
giving \(\Rord(H,K_n)\leq R(G,K_n)\). To obtain such a graph \(G\), we introduce
a novel construction: The vertices of \(G\) are subsets of a ground set \([M]\),
and the fact that \(H\) appears under every ordering of \(G\) comes from Leeb's
canonization theorem for linear orders on \(\binom{[M]}{r}\). At the same time,
\(G\) is sparse in the following sense: \(V(G)\) can be partitioned into
independent sets \(V_1,\ldots,V_k\) such that every vertex of \(V_i\) has at
most \(D=D(\Delta,k)\) neighbors in \(V_1\cup\cdots\cup V_{i-1}\). This allows
us to apply the dependent random choice argument of Alon, Krivelevich and
Sudakov~\cite[Theorem~4.1]{AKS}, and it gives
\(R(G,K_n)=O_G\bigl(n^{(k-1)D}\bigr)\).

In particular, the proof of Theorem~\ref{thm:bounded-degree} gives the following
natural statement, which to the best of our knowledge is new.
\begin{thm}\label{thm:degeneracy}
    For all \(\Delta,k \geq 1\), there exists \(D = D(\Delta,k)\) such that the
    following holds. Let \(H\) be an ordered graph with \(\Delta(H)\leq\Delta\)
    and \(\chi_<(H)\leq k\). Then there is a \(D\)-degenerate unordered graph
    \(G\) such that every ordering of \(G\) contains a copy of \(H\).
\end{thm}

We note that the analogous statement with degeneracy replaced by maximum degree
fails already for \((\Delta,k) = (2,2)\). To see this, let \(P\) be the ordered
path with vertices \(a_1 < \dots < a_m < b_1 < \dots < b_{m-1}\) and edges
\(a_ib_i, a_{i+1}b_i\), \(1 \leq i \leq m-1\). Clearly, \(\Delta(P)=2\) and
\(\chi_<(P) = 2\). Now let \(G\) be any unordered graph with maximum degree
\(D\), and let us show that if \(m > D(D-1)+1\) then there is an ordering of
\(G\) avoiding \(P\). Indeed, let \(G^2\) denote the graph on \(V(G)\) where
\(u,v\) are adjacent if and only if \(G\) has a path of length 2 between \(u\)
and \(v\). Since \(\Delta(G^2) \leq D(D-1)\), there is a proper coloring \(c\)
of \(G^2\) with \(D(D-1)+1\) colors. Now order \(V(G)\) such that the color
classes of \(c\) are intervals in the ordering, with the color labels
increasing. Suppose by contradiction that this ordering has a copy of \(P\), and
let \(v_1 < \dots < v_m\) be the vertices playing the roles of
\(a_1,\dots,a_m\), respectively, in this copy. Since \(b_i\) is a common
neighbor of \(a_i,a_{i+1}\) in \(P\), we have that \(c(v_i) \neq c(v_{i+1})\),
and hence \(c(v_i) < c(v_{i+1})\) by the definition of the ordering (for every
\(1 \leq i \leq m-1\)). But \(c(v_1) < \dots < c(v_m)\) is impossible as \(m >
D(D-1)+1\).

The proof of Theorem~\ref{thm:degeneracy} gives
\[
    D(\Delta,k)\leq k(6\Delta)^{(2\Delta)^{k-2}}.
\]
It would be interesting to determine the best possible dependence on \(\Delta\)
and \(k\).
\begin{que}
    Determine the asymptotic growth of \(D(\Delta,k)\).
\end{que}
\paragraph*{Acknowledgment.} Theorem \ref{thm:forest-bipartite} was proved
jointly with Dingding Dong and Dylan King, and we thank them for many useful
discussions on the subject. The first author is also grateful to Asaf Ferber and
Marcelo Sales for organizing the Desert Discrete Mathematics Workshop (DDMW II),
where preliminary work on this project took place. The workshop was supported by
NSF CAREER Grant DMS-2146406.

\paragraph*{AI Disclosure.} All proofs and constructions were obtained without
AI assistance, except for the example showing that Theorem~\ref{thm:degeneracy}
does not hold for maximum degree (in place of degeneracy), and the parameter
choices in the proof of Lemma~\ref{lem:universal-graph}. For the latter, we
initially left these parameters unspecified and asked GPT-5.6 Sol to determine
them. GPT-5.6 Sol was also used to polish the exposition.

\section{Proof of Theorem~\ref{cor:forest-ramsey}}\label{sec:forest-bipartite}

We will first prove Theorem~\ref{thm:forest-bipartite} and then use it to deduce
Theorem~\ref{cor:forest-ramsey}.

\begin{proof}[Proof of Theorem~\ref{thm:forest-bipartite}]
    By adding edges, we may assume that \(F\) is a tree. If \(n=1\), the result
    is immediate, so assume that \(n\geq2\). Denote the vertices of \(F\) by
    \(x_1<\dots<x_n\). Consider a red-blue edge-coloring of \(K_N\), where
    \(N=1+(n-1)(a+b-1)\). We work in the red graph \(G\), assuming that there is
    no blue ordered copy of \(K_{a,b}\).

    For each \(i\), let \(d_i^-\) and \(d_i^+\) be the numbers of neighbors of
    \(x_i\) to its left and right, respectively (in \(F\)). Partition the vertex
    set of \(K_N\) into intervals \(V_1<\dots<V_n\) with
    \[
        |V_i|=1+(b-1)d_i^-+(a-1)d_i^+.
    \]
    This is possible because
    \[
        \sum_{i=1}^n |V_i|
        =n+(a+b-2)(n-1)
        =N.
    \]

    Let \(x_{i_1},\dots,x_{i_n}\) be a \defn{leaf-elimination ordering} of
    \(F\); that is, for every \(j<n\), the vertex \(x_{i_j}\) has exactly one
    neighbor in \(\{x_{i_{j+1}},\dots,x_{i_n}\}\), which we denote by
    \(x_{p_j}\). (Namely, \(p_j\) is the unique index \(p \in
    \{i_{j+1},\dots,i_n\}\) such that \(x_{i_j}\) is adjacent to \(x_p\).) Such
    an ordering is obtained by fixing a root and listing the vertices in
    non-increasing order of their distance from it.

    \begin{claim}
        There are sets \(U_{i_j}\subseteq V_{i_j}\), \(j=1,\dots,n\), such that
        \(|U_{i_n}|\geq1\) and, for every \(j<n\),
        \[
            |U_{i_j}|\geq
            \begin{cases}
                a,& \text{ if } i_j<p_j,\\
                b,& \text{ otherwise.}
            \end{cases}
        \]
        Moreover, whenever \(k<j\) and \(x_{i_k}x_{i_j}\in E(F)\), every vertex
        of \(U_{i_j}\) has a neighbor in \(U_{i_k}\).
    \end{claim}

    \begin{proofclaim}
        We choose the sets one by one. Since \(x_{i_1}\) is a leaf, the set
        \(U_{i_1}:=V_{i_1}\) has the required size. Now let \(2\leq j\leq n\),
        and suppose that \(U_{i_1},\dots,U_{i_{j-1}}\) have already been chosen.
        Let
        \[
            I:=\{k<j:x_{i_k}x_{i_j}\in E(F)\}\footnote{That is, \(I\) indexes
            the
            children of \(x_{i_j}\).},
        \]
        and define
        \[
            U_{i_j}:=
            \{v\in V_{i_j}:N_G(v)\cap U_{i_k}\neq\varnothing
            \text{ for every }k\in I\}.
        \]
        We lower bound \(|U_{i_j}|\) by subtracting from \(|V_{i_j}|\), for each
        \(k\in I\), the size of \(\{v\in V_{i_j}\colon N_G(v)\cap
        U_{i_k}=\varnothing\}\). We first observe that \(i_j=p_k\) whenever
        \(k\in I\).
        \begin{itemize}
            \item If \(k\in I\) and \(i_k<i_j = p_k\), the induction hypothesis
                gives
                \(|U_{i_k}|\geq a\). Thus at most \(b-1\) vertices of
                \(V_{i_j}\) have no neighbor in \(U_{i_k}\); otherwise, \(a\)
                vertices of \(U_{i_k}\) and \(b\) such vertices of \(V_{i_j}\)
                form a blue ordered \(K_{a,b}\).
            \item Otherwise, \(k\in I\) and \(i_k>i_j = p_k\). Then
                \(|U_{i_k}|\geq b\), so at most \(a-1\) vertices of \(V_{i_j}\)
                have no neighbor in \(U_{i_k}\).
        \end{itemize}

        Consider first the case \(j<n\). Then the set \(I\) (defined above)
        consists of all neighbors of \(x_{i_j}\) except its parent. Therefore,
        \[
            |U_{i_j}|\geq
            \begin{cases}
                |V_{i_j}|-(b-1)d_{i_j}^-
                -(a-1)(d_{i_j}^+-1)=a,
                & \text{if } i_j<p_j,\\
                |V_{i_j}|-(b-1)(d_{i_j}^--1)
                -(a-1)d_{i_j}^+=b,
                & \text{if } i_j>p_j.
            \end{cases}
        \]

        Now consider the case \(j=n\). Then \(I\) consists of all neighbors of
        \(x_{i_n}\), and so
        \[
            |U_{i_n}|
            \geq |V_{i_n}|-(b-1)d_{i_n}^--(a-1)d_{i_n}^+
            =1.
        \]
        The ``moreover'' part of the claim holds by construction.
    \end{proofclaim}

    We now construct a red copy of \(F\) by selecting the vertices of the copy
    in reverse leaf-elimination order. Choose any \(v_{i_n}\in U_{i_n}\). For
    \(j=n-1,\dots,1\), the vertex \(v_{p_j}\) has already been chosen, and the
    claim gives a neighbor \(v_{i_j}\in U_{i_j}\) of \(v_{p_j}\). Every edge of
    \(F\) is therefore represented by a red edge. Moreover, \(v_i\in V_i\) for
    every \(i\), so \(v_1<\dots<v_n\). Thus these vertices form a red ordered
    copy of \(F\), completing the proof.
\end{proof}

\begin{proof}[Proof of Theorem~\ref{cor:forest-ramsey}]
    By adding edges to \(F\), it suffices to prove
    \eqref{eq:forest-off-diagonal} when \(F\) is a tree. We induct\footnote{This
    induction is standard; see the proofs of~\cite[Theorem~2.1]{CFLS}
    and~\cite[Theorem~26]{BCKK}.} on \(k := \chi_<(H)\). If \(k=1\), then \(H\)
    is edgeless, so \(R_<(F,H) \leq m\), as required. Suppose now that \(k \geq
    2\) and that the statement holds for all \(H\) with \(\chi_<(H) < k\).

    Partition \(H\) into two nonempty consecutive induced subgraphs \(H_1 <
    H_2\) such that
    \[
        \chi_<(H_1),\chi_<(H_2)
        \leq \ceil{k/2} < k
    \]
    and let
    \[
        m_1 := |V(H_1)| \qquad m_2 := |V(H_2)| \qquad a := R_<(F,H_1),
        \qquad
        b := R_<(F,H_2).
    \]

    Let \(r := \ceil{\log k}\). Since \(\chi_<(H_i) \leq 2^{r-1}\), the
    induction hypothesis gives
    \begin{equation}\label{eq:forest-bipartite induction}
        a \leq 1 + (n-1)^{r-1}(m_1-1),
        \qquad
        b \leq 1 + (n-1)^{r-1}(m_2-1).
    \end{equation}

    By Theorem~\ref{thm:forest-bipartite}, every red-blue coloring of
    \(K_{1+(n-1)(a+b-1)}\) with no red copy of \(F\) contains a blue
    \(K_{a,b}\). By the definitions of \(a\) and \(b\), its first and second
    parts contain blue copies of \(H_1\) and \(H_2\), respectively. Since all
    edges between the parts are blue, these copies form a blue copy of \(H\).

    By the above, combined with \eqref{eq:forest-bipartite induction} and \(m =
    m_1+m_2\), we conclude that
    \begin{align*}
        R_<(F,H)
        &\leq 1+(n-1)(a+b-1)\\
        &\leq 1+(n-1)+(n-1)^r(m-2)\\
        &\leq 1+(n-1)^r(m-1).\qedhere
    \end{align*}
\end{proof}

\section{Proof of Theorem~\ref{thm:bounded-degree}}

The proof of Theorem~\ref{thm:bounded-degree} uses the following simple
construction.
\begin{construction}\label{con:inclusion-graph}
    Let \(\Delta,k,M\geq 1\) be integers. For \(1 \leq i \leq k\), let
    \[
        V_i=\binom{[M]}{(2\Delta)^{i-1}}.
    \]
    Let \(G=G(M)\) be the graph with vertex set \(V_1\cup\cdots\cup V_{k}\),
    where two distinct vertices are adjacent if and only if they are comparable
    under inclusion (i.e., if one contains the other).
\end{construction}

\begin{lem}\label{lem:universal-graph}
    Let \(\Delta,k\geq 1\) be integers, and let \(H\) be an ordered graph with
    \(\Delta(H)\leq\Delta\) and \(\chi_<(H)\leq k\). There is an integer \(M=M(
    |H|,\Delta,k)\) such that the graph \(G=G(M)\) from
    Construction~\ref{con:inclusion-graph} satisfies the following criteria:
    \begin{enumerate}
        \item for every \(2 \leq i \leq k\), each vertex of \(V_i\) has at most
            \[
                D = k(6\Delta)^{(2\Delta)^{k-2}}
            \]
            neighbors in \(V_1\cup\cdots\cup V_{i-1}\);
        \item every ordering of \(G\) contains an ordered copy of \(H\).
    \end{enumerate}
\end{lem}

Theorem~\ref{thm:bounded-degree} follows easily from this lemma.

\begin{proof}[Proof of Theorem~\ref{thm:bounded-degree}]
    Let \(M=M(|H|,\Delta,k)\) be as in Lemma~\ref{lem:universal-graph}, and set
    \(G=G(M)\). By Item~2, \(\Rord(H,K_n)\leq R(G,K_n)\). Applying the argument
    of Alon, Krivelevich and Sudakov~\cite[Theorem~4.1]{AKS} with the degree
    bound in Item~1 gives\footnote{Their inductive argument requires only a
    bound \(D\) on the number of neighbors in earlier classes, rather than on
    the total degree. Each additional class costs a factor of \(O_G(n^D)\).}
    \[
        R(G,K_n)\leq
        C_H n^{(k-1)D} =
        C_H n^{k(k-1)(6\Delta)^{(2\Delta)^{k-2}}}.\qedhere
    \]
\end{proof}

Thus, it remains to prove Lemma~\ref{lem:universal-graph}. The main difficulty
is Item~2: given an arbitrary ordering, we need to find a subset on which the
ordering has a \emph{canonical} form. For this, we use the following Ramsey-type
result, known as Leeb's canonization theorem (see \cite[Theorem~1.2 and
Lemma~3.8]{RRSSS}).\footnote{Theorem~1.2 of~\cite{RRSSS} gives a subset on which
the ordering satisfies a certain consistency property, while Lemma~3.8 provides
the characterization we use here.}
\begin{thm}\label{thm:leeb-canonization}
    For all positive integers \(r,t\), there exists an integer \(M_0\) such
    that, for every \(M\geq M_0\) and every linear order \(\prec\) on
    \(\binom{[M]}{r}\), there are a set \(X\subseteq[M]\) of size \(t\), a sign
    vector \((\varepsilon_1,\ldots,\varepsilon_r)\in\{-1,1\}^r\), and a
    permutation \(\sigma\) of \([r]\) such that
    \begin{equation}\label{eq:Leeb}
        A\prec B
        \quad\Longleftrightarrow\quad
        (\varepsilon_{\sigma(1)}a_{\sigma(1)},\ldots,
        \varepsilon_{\sigma(r)}a_{\sigma(r)})
        <_{\mathrm{lex}}
        (\varepsilon_{\sigma(1)}b_{\sigma(1)},\ldots,
        \varepsilon_{\sigma(r)}b_{\sigma(r)})
    \end{equation}
    for all \(A=\{a_1<\cdots<a_r\}\) and \(B=\{b_1<\cdots<b_r\}\) in
    \(\binom{X}{r}\), where \(<_{\mathrm{lex}}\) denotes the lexicographic order
    on \(\mathbb{Z}^r\). In this case, we say that the order on \(\binom{X}{r}\)
    is \defn{canonical}.
\end{thm}

Applying Theorem~\ref{thm:leeb-canonization} successively, we may restrict to
\(X\subseteq[M]\) such that the order on \(\binom{X}{(2\Delta)^i}\) is canonical
for every \(0\leq i<k\). The following lemma is the connection between sets of
different sizes. Here and throughout, for nonempty subsets \(A,B\) of a linearly
ordered set under \(\prec\), we write \(A \prec B\) if every element of \(A\) is
smaller than every element of \(B\).\footnote{\(A\prec B\) is vacuously true if
either set is empty.}

\begin{lem}\label{lem:diagonal-selection}
    Let \(\prec\) be a linear order on a set \(\Omega\), and let \(\cA_{i, j}
    \subset \Omega\) be finite for \(i,j\in[k]\). Suppose that
    \[
        \cA_{i,1}\prec \cA_{i,2}\prec\cdots\prec \cA_{i,k}
    \]
    for every \(i\in[k]\). Then there is a permutation \(\tau\) of \([k]\) such
    that
    \[
        \cA_{\tau(1),1}\prec \cA_{\tau(2),2}
        \prec\cdots\prec \cA_{\tau(k),k}.
    \]
\end{lem}
\begin{proof}
    For a set \(\cA\), let \(\ell(\cA)\) (resp. \(r(\cA)\)) denote its smallest
    (resp. largest) element, writing \(\infty\) (resp. \(-\infty\)) when \(\cA\)
    is empty. We choose \(\tau(1),\dots,\tau(k)\) one by one, in order. For each
    \(j=1,\ldots,k\), choose \(\tau(j)\) among the unused indices (i.e., among
    all indices not belonging to \(\{\tau(1),\dots,\tau(j-1)\}\)) such that
    \(r(\cA_{\tau(j),j})\) is smallest. For every \(1\leq j<k\), the index
    \(\tau(j+1)\) was available at step \(j\), so
    \[
        r(\cA_{\tau(j),j}) \preceq r(\cA_{\tau(j+1),j}).
    \]
    Also, by the assumption of the lemma, we have
    \[
        r(\cA_{\tau(j+1),j}) \prec \ell(\cA_{\tau(j+1),j+1}).
    \]
    This proves that
    \[
        \cA_{\tau(j),j}
        \prec\cA_{\tau(j+1),j+1}.\qedhere
    \]
\end{proof}
\begin{proof}[Proof of Lemma~\ref{lem:universal-graph}]
    To ease the notation, let us put
    \[
        r_i := (2\Delta)^{i-1}
    \]
    (for \(1 \leq i \leq k\)), so that \(V_i = \binom{[M]}{r_i}\). Note that
    \(r_i = 2\Delta r_{i-1}\) for \(2 \leq i \leq k\). For Item~1, the classes
    \(V_i\) are independent, and each \(S\in V_i\), \(2 \leq i \leq k\), has at
    most
    \[
        \sum_{j=1}^{i-1}\binom{r_i}{r_j}
        \leq (i-1)\binom{r_i}{r_{i-1}}
        \leq (i-1)(2e\Delta)^{(2\Delta)^{i-2}}
        \leq k(6\Delta)^{(2\Delta)^{k-2}}
    \]
    neighbors in earlier classes.

    For Item~2, fix any linear order \(\prec\) on \(V(G)\). Applying
    Theorem~\ref{thm:leeb-canonization} successively,\footnote{Here we choose
    \(M=M(|H|,\Delta,k)\) sufficiently large.} we obtain \(X\subseteq[M]\) with
    \[
        |X|=k^2|H|(2\Delta)^{k-1}
    \]
    such that, for every \(1 \leq i \leq k\), the order induced by \(\prec\) on
    \(\binom{X}{(2\Delta)^{i-1}}\) is canonical. This means that \eqref{eq:Leeb}
    holds with \(r = r_i\) and for some permutation \(\sigma^{(i)}\) of
    \([r_i]\) and a vector \(\varepsilon^{(i)} \in \{-1,1\}^{r_i}\). Let \(p_i
    := \sigma^{(i)}(1)\) and \(\epsilon_i := \epsilon^{(i)}_{p_i} \in
    \{-1,1\}\). For \(S\in\binom{X}{(2\Delta)^{i-1}}\), let its \defn{anchor}
    \(a(S)\)\footnote{We suppress \(i\) from the notation, since it is
    determined by \(|S|\).} be its \(p_i\)-th smallest element. The key property
    is that for \(S,S'\in\binom{X}{(2\Delta)^{i-1}}\),
    \begin{equation}\label{eq:Leeb order}
        \epsilon_i a(S)<\epsilon_i a(S')
        \quad\Longrightarrow\quad
        S\prec S'.
    \end{equation}
    This holds by considering the first coordinate in \eqref{eq:Leeb}.

    Partition \(X\) into \(k\) consecutive intervals, each of size
    \(k|H|(2\Delta)^{k-1}\), and label them \(X_1,\ldots,X_k\) so that, for each
    \(2\leq i\leq k\),
    \begin{equation}\label{eq:interval-placement}
        X_i<X_1\cup\cdots\cup X_{i-1}
        \text{ if }p_i\leq\frac{r_i}{2},
        \qquad
        X_i>X_1\cup\cdots\cup X_{i-1} \text{ otherwise.}
    \end{equation}
    We may arrange this by choosing \(X_1,\dots,X_k\) in reverse order
    \(i=k,\dots,1\), at each step giving the largest available label to the
    leftmost unlabeled interval if \(p_i\leq r_i/2\), and to the rightmost
    otherwise. Finally, partition each \(X_i\) into \(k\) consecutive intervals
    \(X_{i,1},\ldots,X_{i,k}\), each of size \(|H|(2\Delta)^{k-1}\), numbering
    them from left to right if \(\epsilon_i=1\), and right to left if
    \(\epsilon_i=-1\).

    For \(i,j\in[k]\), let
    \[
        \cA_{i,j}
        =\left\{S\in\binom{X}{r_i}:
        a(S)\in X_{i, j}\right\}.
    \]
    By how the \(X_{i,1},\ldots,X_{i,k}\) were numbered and by the key property
    \eqref{eq:Leeb order}, for every \(i\in[k]\) we have
    \[
        \cA_{i,1}\prec\cA_{i,2}\prec\cdots\prec\cA_{i,k}.
    \]
    By Lemma~\ref{lem:diagonal-selection}, we obtain a permutation \(\tau\) of
    \([k]\) such that
    \begin{equation}\label{eq:diagonal-order}
        \cA_{\tau(1),1}\prec\cA_{\tau(2),2}
        \prec\cdots\prec\cA_{\tau(k),k}.
    \end{equation}

    It will be convenient to work with the inverse of \(\tau\), so put \(\pi :=
    \tau^{-1}\). Then \eqref{eq:diagonal-order} translates to having
    \begin{equation}\label{eq:diagonal-order 2}
        \cA_{i,\pi(i)} \prec \cA_{j,\pi(j)}
        \text{ whenever } \pi(i) < \pi(j).
    \end{equation}

    We are now ready to embed \(H\). Let \(I_1<\cdots<I_k\) be independent
    consecutive intervals partitioning \(V(H)\), allowing empty parts. We
    construct an embedding\footnote{That is, an injective order-preserving graph
    homomorphism.} \(\phi\colon V(H)\to V(G)\) by induction on \(i\), where in
    the \(i\)th step we embed \(I_{\pi(i)}\) into \(\cA_{i,\pi(i)}\). Note that
    \(\cA_{i,\pi(i)} \subseteq \binom{X}{r_i}\). This means that at step \(i\)
    we use sets of size \(r_i\). The embedding idea is simple: edges of \(G\)
    are given by inclusion. Thus, we need to make sure that the image
    \(\phi(v)\) of a vertex \(v\) contains the images of all already-embedded
    neighbours of \(v\). This is why we embed in increasing order of set size
    (i.e., we use sets of size \(r_i\) in the \(i\)th step). In addition to
    these ``mandatory" elements, we also add additional elements to \(\phi(v)\)
    to guarantee that the internal order inside each \(I_i\) is preserved.

    At each step, we maintain the following invariants for the vertices already
    embedded:
    \begin{enumerate}
        \item[(a)] For every \(i \in [k]\) and
            \(v\in I_{\pi(i)}\), it holds that \(\phi(v)\subseteq
            X_1\cup\cdots\cup X_i\) and \(\phi(v)\in\cA_{i,\pi(i)}\);
        \item[(b)] whenever \(v<w\) belong to the same interval \(I_i\),
            \(\phi(v)\prec\phi(w)\);
        \item[(c)] for all \(1 \leq j < i \leq k\) and every edge \(uv\in E(H)\)
            with \(u\in I_j\) and \(v\in I_i\), it holds that
            \(\phi(u)\phi(v)\in E(G)\).
    \end{enumerate}
    Once these properties are established, Item~(a) and
    \eqref{eq:diagonal-order 2} show that the order \(I_1<\cdots<I_k\) is
    preserved by \(\phi\), while
    Item~(b) preserves the order within each \(I_i\). Together, these imply that
    \(\phi\) preserves the vertex order and is injective. Since each \(I_i\) is
    independent, Item~(c) says that \(\phi\) preserves every edge, completing
    the proof.

    First, for each \(i\in[k]\), choose pairwise disjoint intervals
    \(R_v\subseteq X_{i,\pi(i)}\), \(v\in I_{\pi(i)}\), each of size \(r_i\),
    such that \(\epsilon_i R_v<\epsilon_i R_w\) whenever \(v<w\) in
    \(I_{\pi(i)}\) (namely, if \(\epsilon_i = 1\) then the order among the
    intervals \(R_v\) matches the order on \(I_{\pi(i)}\), and if \(\epsilon_i =
    -1\) then the order is reversed). Note that we can choose such disjoint
    intervals \(R_v \subseteq X_{i,\pi(i)}\) because
    \[
        |I_{\pi(i)}| r_i = |I_{\pi(i)}|(2\Delta)^{i-1}
        \leq |H|(2\Delta)^{k-1}
        =|X_{i,\pi(i)}|.
    \]

    Now let \(i\in[k]\), and suppose that \(I_{\pi(j)}\) has been embedded for
    every \(1 \leq j < i\). We embed \(I_{\pi(i)}\). For each \(v \in
    I_{\pi(i)}\), let
    \[
        U_v=\bigcup_{\substack{j<i, \;
                              u\in I_{\pi(j)} \\
                              uv\in E(H)}}\phi(u)
    \]
    be the union of the images of its already embedded neighbours. By induction,
    \[
        |U_v|\leq
        \Delta \cdot r_{i-1} \leq
        \frac{1}{2}r_i.
    \]
    (Here, for \(i=1\) we set \(r_0 := 0\).) So we may define \(\phi(v)\) by
    adding \(r_i-|U_v| \ge r_i/2\) points of \(R_v\) to \(U_v\). This is
    possible as \(|R_v| = r_i\), and guarantees that \(|\phi(v)| = r_i\).

    It remains to verify the three invariants. For Item~(a), we have
    \(R_v\subseteq X_{i,\pi(i)} \subset X_{i}\) and, by Item~(a) at steps
    \(1,\ldots,i-1\), we have \(U_v\subseteq X_1\cup\cdots\cup X_{i-1}\).
    Therefore, \(R_v \cap U_v = \varnothing\) and by construction
    \[
        \phi(v)\subseteq X_1\cup\cdots\cup X_i,
        \qquad
        |\phi(v)|=r_i.
    \]
    It remains to verify that \(a(\phi(v))\in X_{i,\pi(i)}\) (so that
    \(\phi(v)\in\cA_{i,\pi(i)}\)). This follows from the following claim
    (because \(R_v \subset X_{i,\pi(i)}\)).

    \begin{claim}\label{claim:anchor-location}
        \(a(\phi(v))\in R_v\).
    \end{claim}
    \begin{proofclaim}
        If \(i = 1\) then \(\phi(v) \subseteq R_v\) so the claim is immediate.
        Now let \(i \geq 2\). Since \(R_v\subseteq X_i\) and \(U_v\subseteq
        X_1\cup\cdots\cup X_{i-1}\), \eqref{eq:interval-placement} gives
        \[
            R_v<U_v
            \quad
            \text{if }p_i\leq\frac{r_i}{2},
            \qquad
            R_v>U_v
            \quad\text{otherwise}.
        \]
        In the first case, \(a(\phi(v))\) is among the \(r_i/2\) smallest
        elements of \(\phi(v)\) and \(\phi(v)\cap R_v\) contains the
        \(|\phi(v)\cap R_v|\) smallest elements of \(\phi(v)\). Since
        \[
            |\phi(v)\cap R_v|\geq\frac{r_i}{2},
        \]
        it follows that \(a(\phi(v))\in R_v\). The other case is analogous.
    \end{proofclaim}

    Item~(b) follows from Claim~\ref{claim:anchor-location}, the ordering of the
    \(\{R_v \colon v \in I_{\pi(i)}\}\), and the canonical order on
    \(\binom{X}{r_i}\). Indeed, consider any \(v,w \in I_{\pi(i)}\) with \(v <
    w\). Suppose first that \(\epsilon_i = 1\). Then \(R_v < R_w\). Also, as
    \(a(\phi(v))\in R_v\) and \(a(\phi(w))\in R_w\), \eqref{eq:Leeb order} gives
    \(\phi(v) \prec \phi(w)\). The case that \(\epsilon_i = -1\) is analogous
    (with \(R_v > R_w\)).

    Finally, Item~(c) holds because \(\phi(v)\) contains \(\phi(u)\) for every
    previously embedded neighbour \(u\) of \(v\), so \(\phi(v) \phi(u) \in
    E(G)\) according to the adjacency rule in \(G\).
\end{proof}

\fontsize{11pt}{12pt}
\selectfont

\hypersetup{linkcolor={red!70!black}}
\setlength{\parskip}{2pt plus 0.3ex minus 0.3ex}

\bibliographystyle{auxfile}
\bibliography{bib}

\end{document}